\documentclass[12pt,final]{article}
\usepackage{amsmath,amsthm,amsfonts,amssymb,graphicx,enumerate,psfrag}
\usepackage{mathrsfs}
\usepackage{fullpage}

\usepackage[colorlinks,citecolor=blue,urlcolor=blue]{hyperref}
\usepackage[utf8]{inputenc} 
\newtheorem{lemma}{Lemma}
\newtheorem{theorem}{Theorem}

\newtheorem{corollary}{Corollary}

\newcommand{\dN}{\mathbb {N}}

\newcommand{\dR}{\mathbb {R}}
\newcommand{\vv}{\mathfrak {v}}
\newcommand{\lip}{\textsc{lip}}

\newcommand{\EE}{{\mathbb{E}}}

\newcommand{\PP}{{\mathbf{P}}}

\newcommand{\cP}{\mathcal{P}}

\newcommand{\cX}{\mathcal {X}}
\newcommand{\cI}{\mathcal {I}}

\newcommand{\cW}{\mathcal {W}}
\newcommand{\dtv}{d_{\textsc{tv}}}

\title{Spectral gap and curvature of monotone Markov chains}
\author{Justin Salez\footnote{CEREMADE, CNRS, UMR 7534, Université Paris-Dauphine, PSL University, 75016 Paris, France}}
\begin{document}

\maketitle
\begin{abstract}
We prove that the absolute spectral gap of any monotone Markov chain coincides with its optimal Ollivier-Ricci curvature, where the word \emph{optimal} refers to the choice of the underlying metric. Moreover, we provide a new expression in terms of local variations of increasing functions, which has several practical advantages over the traditional variational formulation using the Dirichlet form.  As an illustration, we explicitly determine the optimal curvature and spectral gap of the non-conservative exclusion process with heterogeneous reservoir densities on any network, despite the lack of reversibility. 
\end{abstract}

\section{Introduction}

\paragraph{Spectral gap.} Throughout the paper, we fix an irreducible  stochastic matrix $P$ on a finite state space $\cX$. Among the various parameters that are traditionally used to quantify the mixing properties of $P$, the most basic one is arguably the (absolute) \emph{spectral gap}
\begin{eqnarray*}
\gamma_\star & = & 1-\max\left\{|\lambda|\colon \lambda\textrm{ is a non-trivial eigenvalue of }P\right\},
 \end{eqnarray*} 
where \emph{non-trivial}  means that constant eigenfunctions are excluded. This quantity classically captures the rate of convergence to equilibrium of the chain, in the sense that
\begin{eqnarray*}
\left(\dtv(t)\right)^{1/t} & \xrightarrow[t\to\infty]{} & 1-\gamma_\star,
\end{eqnarray*}
where $\dtv(t)$ denotes the worst-case total-variation distance to stationarity at time $t$. 
Under the usual assumption that $P$ is lazy and reversible, the spectral gap coincides with the Poincaré constant and hence provides universal controls on mixing times,  concentration and isoperimetry. We refer the unfamiliar reader to the classical textbooks \cite{MR3726904, MR2341319} for details.

\paragraph{Ollivier-Ricci curvature.} Let us now assume that $\cX$ is equipped with a metric $\rho$. When the support of $P$ is symmetric, a canonical choice is of course the \emph{hop-count distance} 
\begin{eqnarray}
\label{def:default}
\rho(x,y) & := & \min\{t\ge 0\colon P^t(x,y)>0\},
\end{eqnarray}
but more sophisticated variants can prove fruitful, as we shall see. We can then use optimal transport to   `lift' the metric $\rho$ to the distributional level. Specifically,  the \emph{Wasserstein distance} between two probability measures $\mu$ and $\nu$ on $\cX$ is defined as follows:
\begin{eqnarray*}
\cW\left(\mu,\nu\right) & := & \min_{X\sim\mu,Y\sim\nu}\EE\left[\rho(X,Y)\right],
\end{eqnarray*}
where the minimum runs over all possible couplings $(X,Y)$ of $\mu$ and $\nu$. Minimizers always exists by  compactness, and will later be referred to as \emph{optimal} couplings. 
Following Ollivier \cite{ollivier2009ricci,ollivier2010survey}, we then define the \emph{curvature} of $P$  to be  the largest number $\kappa\in\dR$ such that 
\begin{eqnarray}
\label{def:curvature}
\forall \mu,\nu\in\cP(\cX),\qquad \cW(\mu P,\nu P) & \le & (1-\kappa)\cW(\mu,\nu).
\end{eqnarray}
By Kantorovich duality, this is in fact equivalent to the functional inequality
\begin{eqnarray*}
\forall f\in {\mathbb C}^{\cX},\qquad \|Pf\|_{\lip} & \le & (1-\kappa)\|f\|_{\lip},
\end{eqnarray*}
 where $\|f\|_\lip:=\max_{x\ne y}\frac{|f(x)-f(y)|}{\rho(x,y)}$ denotes the Lipschitz constant of $f$.
In particular, choosing $f$ to be a non-constant eigenfunction of $P$ readily shows that
\begin{eqnarray}
\label{ineq}
\kappa & \le & \gamma_\star.
\end{eqnarray}
In words,  the curvature  always constitutes a lower bound on the absolute spectral gap. More importantly, this  fundamental quantity has been shown to provide an array of powerful controls in terms of geometry \cite{jost2019Liouville,munch2019non}, mixing times \cite{646111,MR2316551}, expansion \cite{salez2021sparse, Muench2023}, concentration of measure \cite{MR2683634,joulin2007poisson,eldan2017transport}, spectral independence \cite{MR4415182} and even the cutoff phenomenon \cite{salez2021cutoff}.

\paragraph{Optimal curvature.} An elementary but crucial observation  is that the left-hand side  of the inequality (\ref{ineq}) depends on the choice of the underlying metric $\rho$, whereas the right-hand side does not. This asymmetry can obviously be turned to one's advantage, as follows: instead of blindly setting $\rho$ to the default choice (\ref{def:default}),  one can treat it as a variable which one can try to fine-tune so as to optimize the bound (\ref{ineq}). Beyond the spectral gap, this strategy can be applied to boost any estimate that uses the Ollivier-Ricci curvature as an input. It has already led to remarkable improvements in the mixing-time analysis of several important Markov chains, and we refer the interested reader to the seminal papers \cite{MR1716763,Bordewich2005MetricCS} for an introduction to this   line of research. Motivated by this fertile idea, we here define the \emph{optimal curvature} of the chain to be the number
\begin{eqnarray*}
\kappa_\star & := & \sup_{\rho}\kappa(\rho),
\end{eqnarray*}
where $\kappa(\rho)$ denotes the curvature of $P$ with respect to the metric $\rho$, and where the supremum runs over all possible metrics $\rho$ on $\cX$. While highly meaningful from a  theoretical viewpoint, this \emph{intrinsic} notion of curvature may seem difficult to manipulate in practice: its definition involves three nested optimization problems that a priori require a test over all pairs of probability measures $(\mu,\nu)$, all couplings $(X,Y)$ of $\mu P$ and $\nu P$, and all metrics $\rho$ on $\cX$. However, our main discovery is that $\kappa_\star$ admits a remarkably simple expression as soon as $P$ satisfies a classical and widespread property known as \emph{monotony}.

\paragraph{Monotone chains.} From now on, we fix a partial order $\preceq$ on our state space $\cX$. To avoid degeneracies, we assume that this relation is rich enough to \emph{connect} $\cX$ in the usual sense: for any states $x,y\in\cX$, there is a finite sequence $(x_0,\ldots,x_\ell)$ such that $x_0=x$, $x_\ell=y$, and for all $1\le i\le \ell$,  $x_{i-1}\preceq x_{i}$ or $x_{i}\preceq x_{i-1}$. We let $\cI$ denote the cone of functions $f\colon\cX\to\dR$ which are  \emph{increasing}, meaning that $f(x)\le f(y)$ whenever $x\preceq y$. Finally, we say that our matrix $P$ is \emph{monotone} if its action preserves the set $\cI$, i.e.
\begin{eqnarray*}
f\in\cI & \Longrightarrow & Pf\in\cI.
\end{eqnarray*} By Strassen's Theorem, this stability property is equivalent to the existence of a \emph{monotone Markovian coupling}: given any pair $(x,y)\in\cX^2$ with $x\preceq y$, we can jointly construct two random states $X$ and $Y$ with respective marginal distributions $P(x,\cdot)$ and $P(y,\cdot)$ such that $\PP(X\preceq Y)=1$. Such chains are ubiquitous in statistical physics: emblematic examples include  Glauber dynamics for the ferromagnetic Ising model, lazy birth-and-death chains, the noisy voter model, or the non-conservative exclusion process with arbitrary reservoir densities. We refer to  \cite[Chapter 22]{MR3726904} for several other examples, as well as many interesting properties. The purpose of the present paper is to  reveal a remarkable connection between the curvature and the spectral gap of such chains. 
\section{Main result and implications}
Our main contribution is the following surprising result, which asserts that monotone chains always saturate the inequality (\ref{ineq}): their optimal curvature $\kappa_\star$ equals their spectral gap $\gamma_\star$. Moreover, we provide a remarkably simple characterization of those constants in terms of \emph{local variations} of increasing functions. Specifically, we introduce the quantity
 \begin{eqnarray}
\vv_\star & := &  \inf_{f\in\mathring{\cI}}\max_{x\vartriangleleft y}\left\{\frac{Pf(y)-Pf(x)}{f(y)-f(x)}\right\},
 \end{eqnarray}
where $\mathring{\cI}$ denotes the set of  \emph{strictly increasing} functions  $f\colon\cX\to\dR$ (meaning that $f(x)< f(y)$ whenever $x\prec y$), and where $\vartriangleleft$ denotes the usual \emph{covering} relation on $\cX$:
\begin{eqnarray}
\label{def:covering}
x\vartriangleleft y & \Longleftrightarrow & x\prec y \textrm{ and there is no } z\in\cX\textrm{ such that } x\prec z\prec y. 
 \end{eqnarray} 
\begin{theorem}[Spectral gap and curvature of monotone chains]\label{th:main}If $P$ is monotone, then
\begin{eqnarray*}
 \kappa_\star & = & \gamma_\star \ = \ 1-\vv_\star.
\end{eqnarray*}
\end{theorem}

\paragraph{Discussion and perspectives.} Beyond the obvious theoretical interest of this surprising equivalence between three fundamentally different notions (a spectral one, a geometric one, and a variational one), we believe that the above result is of practical interest for several reasons. An immediate one is that curvature classically provides much more powerful controls  on mixing times and on concentration than  the spectral gap does: our first equality $\kappa_\star=\gamma_\star$ shows that those improvements come `for free' in all monotone chains. Another potentially rich source of applications lies in our second equality $\gamma_\star \ = \ 1-\vv_\star$. To see why, let us recall that the spectral gap of lazy reversible chains admits the variational characterization
\begin{eqnarray}
\gamma_\star & = & \inf_{f}\left\{\frac{\sum_{x,y\in\cX}\pi(x)P(x,y)\left(f(y)-f(x)\right)^2}{\sum_{x,y\in\cX}\pi(x)\pi(y)\left(f(y)-f(x)\right)^2}\right\},
\end{eqnarray}
where $\pi$ denotes the stationary law and where the infimum runs over all non-constant test functions $f\colon\cX\to\dR$. Compared to this classical formulation, our variational characterization in terms of $\vv_\star$ offers a number of practical advantages, which we now enumerate.
\begin{enumerate}
\item It does not require reversibility, nor laziness.
\item It involves a supremum instead of an infimum, so that any test function can be plugged into  the formula to produce a lower bound on the spectral gap.
\item It is purely local instead of involving an average under the stationary law $\pi$, which is sometimes very far from explicit.
\item It is suitable for comparing very different chains: for example, we see that  $\gamma_\star(P)\ge \gamma_\star(Q)$ as soon as $Q$ is `more increasing' than $P$ in the sense that $Q-P$ is monotone.
\end{enumerate}
We intend to implement those ideas on concrete models in the near future. Let us here only mention a very simple application of Theorem \ref{th:main} to strictly increasing eigenfunctions.
\paragraph{Strictly increasing eigenfunctions.} Suppose that $P$ is monotone and that it admits an eigenfunction $f$ in $\mathring{\cI}$, associated with some eigenvalue $\lambda$. We can then plug $f$ into the definition of $\vv_\star$ and use our variational characterization to deduce that $\gamma_\star\ge 1-\lambda$, while the  converse inequality is trivially forced by the very definition of $\gamma_\star$. Thus, we can safely conclude that $\gamma_\star=1-\lambda$, without having to  search for other eigenfunctions and their associated eigenvalues.  Let us record this useful observation here. 
\begin{corollary}[Strictly increasing eigenfunctions]\label{co:eigen}If $P$ is monotone and if $f\in\mathring \cI$ satisfies $Pf=\lambda f$ for some $\lambda\in\dR$, then we necessarily have $\lambda=1-\gamma_\star$. 
\end{corollary}
To the best of our knowledge, this result was only known to hold under the additional assumptions that $P$ is reversible and that its support only consists of pairs of comparable elements \cite[Lemma 22.18]{MR3726904}. An example of a monotone chain which fails to satisfy those two additional requirements is the non-conservative exclusion process with inhomogeneous reservoir densities. In Section \ref{sec:exclusion}, we will use the above corollary to explicitly determine the  spectral gap and optimal curvature of this process on any network. Let us first recast our main result in the  setting of continuous-time Markov chains.

\paragraph{Continuous-time version.} Consider an irreducible Markov generator $L$ on the poset $(\cX,\preceq)$, and write $(P_t)_{t\ge 0}$ for the associated semi-group. The spectral gap $\gamma_\star$ is the smallest real part of a non-zero eigenvalue of $-L$, and the optimal curvature $\kappa_\star$ is defined as before, except that the requirement (\ref{def:curvature}) is replaced with its continuous-time analogue:
\begin{eqnarray*}
\forall t\ge 0,\qquad \cW\left(\mu P_t,\nu P_t\right) & \le & e^{-\kappa t}\cW(\mu,\nu).
\end{eqnarray*}
We naturally call the generator $L$  \emph{monotone} if the cone of increasing functions $\cI$ is preserved under the action of the semi-group. In other words, $L$ is monotone if the stochastic matrix $P_t$ is monotone in the previous sense, for each $t>0$.  This allows us to apply our main theorem to $P_t$ for every $t>0$. We may then send $t\to 0$ and exploit the approximation $P_t=I+tL+o(t)$ to obtain the following continuous-time version of Theorem \ref{th:main} and Corollary \ref{co:eigen}.
\begin{corollary}[Continuous-time version]\label{co:continuous}If $L$ is monotone, then 
\begin{eqnarray*}
\kappa_\star & = & \gamma_\star \ = \  -\inf_{f\in\mathring{\cI}}\max_{x\vartriangleleft y}\left\{\frac{Lf(y)-Lf(x)}{f(y)-f(x)}\right\}.
\end{eqnarray*}
In particular, if $f\in\mathring\cI$ satisfies $-Lf=\lambda f$ for some $\lambda\in \dR$, then we necessarily have $\lambda=\gamma_\star$. 
\end{corollary}

\section{Proof of the main result}
We henceforth assume  that $P$ is monotone, and our goal is to establish the double identity
\begin{eqnarray*}
\gamma_\star & = & 1-\vv_\star \ = \ \kappa_\star.
\end{eqnarray*}
Since we always have $\kappa_\star\le \gamma_\star$,  our task reduces to the following two inequalities:
\begin{enumerate}[(i)]
\item $\gamma_\star\le 1-\vv_\star$;
\item $\kappa_\star \ge 1-\vv_\star$.
\end{enumerate}
We now prove each of those two claims separately. 
\subsection{Proof of Claim (i)}
The first claim is trivial if $P$ admits an eigenfunction $f$ in $\mathring\cI$, since the latter may then be simply plugged into the definition of $\vv_\star$. Unfortunately, not all monotone chains admit an eigenfunction in $\mathring \cI$: a simple counter-example is the product of two chains on $\{0,1\}$ with different spectral gaps.  Our solution consists in considering `approximate eigenvectors' in the sense of the following lemma, which establishes Claim (i) above. 
\begin{lemma}[Proving Claim (i)]For each $\varepsilon>0$, there is $f_\varepsilon\in\mathring{\cI}$ such that 
\begin{eqnarray}
\max_{x\vartriangleleft y}\left\{\frac{Pf_\varepsilon(y)-Pf_\varepsilon(x)}{f_\varepsilon(y)-f_\varepsilon(x)}\right\}  & \le & 1-\gamma_\star+\varepsilon.
\end{eqnarray}
\end{lemma}
\begin{proof}
We first note that the set $\mathring\cI$ is not empty. Indeed, a strictly increasing function $f\colon\cX\to\dR$ can be constructed sequentially by choosing any minimal element, assigning it a value that is strictly larger than all previously assigned ones, removing it from the poset and iterating. Upon replacing $f$ by $f-\pi f$ if necessary, we may further assume that $f$ is centered under the stationary law $\pi$. Now, fix $\varepsilon>0$ and consider the centered function
\begin{eqnarray*}
f_\varepsilon & := & \left({\rm I}-\frac{P}{1-\gamma_\star+\varepsilon}\right)^{-1}f \ = \ \sum_{n=0}^\infty\left(\frac{1}{1-\gamma_\star+\varepsilon}\right)^n P^n f,
\end{eqnarray*}
whose existence is guaranteed by the fact that the restriction of $P$ to centered functions has spectral radius $1-\gamma_\star$. Note that $f_\varepsilon$ is strictly increasing, because each term in the above series is increasing and the first one is strictly increasing. Moreover, we have $Pf_\varepsilon =  (1-\gamma_\star+\varepsilon)(f_\varepsilon-f)$ by construction, so that for any $x\vartriangleleft y$,
\begin{eqnarray*}
Pf_\varepsilon(y)-Pf_\varepsilon(x) & = & (1-\gamma_\star+\varepsilon)\left[\left(f_\varepsilon(y)-f_{\varepsilon}(x)\right)-\left(f(y)-f(x)\right)\right]\\
& \le & (1-\gamma_\star+\varepsilon)\left(f_\varepsilon(y)-f_{\varepsilon}(x)\right),
\end{eqnarray*}
as desired. 
\end{proof}
\subsection{Proof of Claim (ii)}

Our proof of Claim (ii) relies on the simple observation that the curvature $\kappa$ of a monotone chain admits a remarkably simple expression when the underlying metric $\rho$ has a certain `gradient structure', which we now describe. Recall that the \emph{Hasse diagram}  of the poset $(\cX,\preceq)$ is the graph $G=(\cX,E)$ whose edge-set is obtained by symmetrizing  the covering relation $\vartriangleleft$ defined at (\ref{def:covering}):
\begin{eqnarray*}
E & := & \left\{(x,y)\in\cX^2\colon x\vartriangleleft y\textrm{ or }y\vartriangleleft x\right\}.
\end{eqnarray*}
As usual, a \emph{path} in $G$ is a finite sequence of vertices $\sigma=(\sigma_0,\ldots,\sigma_\ell)$ such that $\{\sigma_{i-1},\sigma_i\}\in E$ for all $1\le i\le \ell$. The integer $\ell\in\{0,1,2,\ldots,\}$ is called the length of the path, and denoted by $|\sigma|$. We write $\Sigma_{x\to y}$ for the set of all paths  starting at $x$ and ending at $y$. Now, given any function $f\in \mathring{\cI}$, we define a metric $\rho$ on $\cX$ by the following formula: 
\begin{eqnarray}
\label{def:rho}
\rho(x,y) & := & \min_{\sigma\in\Sigma_{x\to y}}\left\{\sum_{i=1}^{|\sigma|}|f(\sigma_i)-f(\sigma_{i-1})|\right\}.
\end{eqnarray}
On the resulting metric space $(\cX,\rho)$, the curvature of any monotone chain is given by the following lemma, from which Claim (ii) readily follows by taking a supremum over all $f\in\mathring\cI$. 
\begin{lemma}[Proving Claim (ii)]The curvature of $P$ with respect to the metric (\ref{def:rho}) is 
\begin{eqnarray*}
\kappa(\rho) & = & 1-\max_{x\vartriangleleft y}\left\{\frac{Pf(y)-Pf(x)}{f(y)-f(x)}\right\}.
\end{eqnarray*}
\end{lemma}
\begin{proof}
It is clear from the definition (\ref{def:rho}) that our metric $\rho$ satisfies
\begin{eqnarray}
\rho(x,y) & \ge & f(y)-f(x),
\end{eqnarray}
for all $x,y\in\cX$. Moreover, there is equality whenever $x\preceq y$, since the very definition of the covering relation $\vartriangleleft$ then guarantees the existence of a path $\sigma$ from $x$ to $y$ such that $\sigma_{i-1}\vartriangleleft \sigma_i$ for all $1\le i \le |\sigma|$. Now, fix $x,y\in\cX^2$ such that $x \vartriangleleft y$. Then, for any coupling $(X,Y)$ of $P(x,\cdot),P(y,\cdot)$, the above observation ensures that 
\begin{eqnarray}
\EE\left[\rho(X,Y)\right] & \ge & \EE\left[f(Y)-f(X)\right] \ = \ Pf(y)-Pf(x),
 \end{eqnarray} 
with equality as soon as $\PP(X\preceq Y)=1$. But our monotony assumption on $P$ precisely guarantees the existence of such a monotone coupling. Thus, we obtain the two identities
\begin{eqnarray*}
{\cW\left(P(x,\cdot),P(y,\cdot)\right)} & = &  Pf(y)-Pf(x),\\
\rho(x,y) & = & f(y)-f(x),
\end{eqnarray*}
whenever $x\vartriangleleft y$. This already shows that  $\vv:=\max_{x\vartriangleleft y}\left\{\frac{Pf(y)-Pf(x)}{f(y)-f(x)}\right\}$ is the smallest number such that the inequality
\begin{eqnarray}
\label{Dirac}
{\cW\left(P(x,\cdot),P(y,\cdot)\right)}\le \vv\, \rho(x,y).
\end{eqnarray}
holds whenever $\{x,y\}\in E$. Now, observe that this last inequality automatically extends to arbitrary points $x,y\in\cX$, since we can apply it to all edges along a \emph{geodesic} path from $x$ to $y$ (i.e., a path achieving the minimum in the definition of $\rho(x,y)$) and then invoke the triangle inequality for $\cW\left(\cdot,\cdot\right)$ to conclude. Finally, given $\mu,\nu\in\cP(\cX)$, we can construct a coupling $(X',Y')$ of $\mu P$ and $\nu P$ in two steps as follows:
\begin{enumerate}
\item Generate  $(X,Y)$ according to an optimal coupling of $\mu$ and $\nu$.
\item For each possible realization $(x,y)\in\cX^2$, conditionally on the event $\left\{(X,Y)=(x,y)\right\}$, let $(X',Y')$ be distributed as an optimal coupling of $P(x,\cdot)$ and $P(y,\cdot)$. 
\end{enumerate} 
The resulting pair $(X',Y')$ is clearly  a coupling of $\mu P$ and $\nu P$ and therefore, 
\begin{eqnarray*}
\cW(\mu P,\nu P) & \le & \EE\left[\rho(X',Y')\right] \\
& = & \sum_{x,y\in\cX}\PP\left((X,Y)=(x,y)\right)\cW\left(P(x,\cdot),P(y,\cdot)\right)\\
& \le & \vv\sum_{x,y\in\cX}\PP\left((X,Y)=(x,y)\right)\rho(x,y)\\
& = &  \vv\EE\left[\rho(X,Y)\right]\\
& = & \vv\cW(\mu,\nu),
\end{eqnarray*}
where the middle inequality is simply (\ref{Dirac}). Since this is true for all $\mu,\nu\in\cP(\cX)$, we may safely conclude from this computation that $1-\vv\le \kappa(\rho)$, and the optimality of $\vv$ in (\ref{Dirac}) shows that this is in fact an equality, as desired.  
\end{proof}

\section{Application to the exclusion process} 
\label{sec:exclusion}
The exclusion process is a classical model of interacting random walks in which indistinguishable particles attempt to evolve independently on a graph, except that their jumps are canceled if the destination is already occupied \cite{MR268959,MR2108619}. 
 Here we consider the non-conservative variant of the model, where particles may additionally be created and annihilated at certain vertices, modeling contact with external reservoirs. We refer the reader to the seminal works \cite{MR1997915,MR2437527} or the more recent papers \cite{gantert2021mixing,goncalves2021sharp,MR4546624,tran2022cutoff} for background and motivation. The model is parametrized by an integer $n\in\dN$ representing the number of sites, a non-negative symmetric array  $(c_{ij})_{1\le i,j\le n}$ specifying the exchange rates, and two non-negative vectors $(r^+_i)_{1\le i\le n}$ and $(r^-_i)_{1\le i\le n}$ describing the creation and annihilation rates. 
The exclusion process with those parameters is the continuous-time Markov chain with state space $\cX=\{0,1\}^n$ and whose infinitesimal generator  $L$ acts on any $f\colon \cX\to\dR$ as follows:
\begin{eqnarray*}
(L f)(x) & := & \frac 12 \sum_{i=1}^n\sum_{j=1}^n c_{ij}\left[f(x^{i\leftrightarrow j})-f(x)\right] \\ &  & +\sum_{i=1}^nr^+_i\left[f(x^{i,1})-f(x)\right]
\\ & &  +\sum_{i=1}^nr^-_i\left[f(x^{i,0})-f(x)\right],
\end{eqnarray*}
where $x^{i\leftrightarrow j}$, $x^{i,1}$ and $x^{i,0}$ naturally denote the vectors obtained from $x$ by respectively swapping the $i-$th and $j-$th coordinates (exchange), resetting the $i-$th coordinate to $1$ (creation), and resetting it to  $0$ (annihilation). To ensure irreducibility, we henceforth assume that the support of the matrix $c$ connects the sites, and that the supports of $r^+$ and $r^{-}$ are not empty. 

A simple object that plays an important role in the analysis of the above process is the \emph{Laplace matrix} $\Delta=(\Delta_{ij})_{1\le i,j\le n}$ of the underlying network, defined as follows:
\begin{eqnarray*}
\Delta_{ij} & := &
\left\{
\begin{array}{ll}
c_{ij} & \textrm{if }j\ne i\\
-r^+_i-r^-_i-\sum_{k\in[n]\setminus\{i\}}c_{ik} & \textrm{if }j=i.
\end{array}
\right.
\end{eqnarray*}
This  symmetric negative definite matrix describes the evolution of a single random walker moving across the $n$ sites according to the conductances $(c_{ij})_{1\le i,j\le n}$ and killed at the space-varying rates $(r^+_i+r^-_i)_{1\le i\le n}$. Our Theorem \ref{th:exclusion} below shows that the mixing properties of the $2^n\times 2^n$ generator $L$ are essentially dicated by those of the $n\times n$ matrix $\Delta$. This remarkable `dimensionality reduction' constitutes a non-conservative extension of  Aldous' celebrated \emph{spectral gap conjecture}, now proved by Caputo, Liggett and Richthammer  \cite{MR2629990}. Interestingly, a similar phenomenon has recently been established in several other interacting particle models \cite{MR2429910,MR3984254,MR4152648,caputo2020spectral,quattropani2021mixing,bristiel2021entropy, kim2023spectral}. However, to the best of our knowledge, its validity for the exclusion process was only known to hold in the reversible case where the vector $\kappa^++\kappa^-$ is constant \cite{MR4546624}. 
\begin{theorem}[Dimensionality reduction]\label{th:exclusion}
The spectral gap and optimal curvature of the generator $L$ coincide with the smallest eigenvalue of the  matrix $-\Delta$. 
\end{theorem}

\begin{proof}

Let $\lambda>0$ be the smallest eigenvalue of the definite positive matrix $-\Delta$. By Perron's theorem and our irreducibility assumption, there is a corresponding eigenvector $v=(v_i)_{1\le i\le n}$ whose entries are all positive. Set $p:=(-\Delta)^{-1} r^+$, and consider the function $f\colon\{0,1\}^n\to\dR$ defined by
\begin{eqnarray*}
f(x) & := & \langle v,x-p\rangle \ = \ \sum_{i=1}^n v_i(x_i-p_i).  
\end{eqnarray*}
We claim that $f$ is an eigenfunction of $L$ corresponding to the eigenvalue $-\lambda$. Indeed, using the definition of $L$, we find 
\begin{eqnarray*}
(Lf)(x) & = & \sum_{i=1}^nv_i\left[\sum_{j=1}^nc_{ij}(x_j-x_i)+r^+_i(1-x_i)-r^-_ix_i\right]\\
&  = & \langle v,\Delta x +r^+\rangle\\
&  = & \langle v,\Delta(x-p) \rangle\\
& = &  \langle \Delta v,x-p \rangle\\
& = & -\lambda f(x),
\end{eqnarray*}
where we have successively used the identity $r^+=-\Delta p$, the symmetry of $\Delta$ and  the identity $\Delta v=-\lambda v$.  Now, the positivity of $v$ ensures that the eigenfunction $f$ is strictly increasing with respect to  the usual coordinate-wise order $\preceq$ on $\{0,1\}^n$. Moreover, this order is clearly preserved under each  possible transition $x\mapsto x^{i\leftrightarrow j}$, $x\mapsto x^{i,1}$ and $x\mapsto x^{i,0}$, and hence also under the `obvious' coupling that consists in performing those transitions at the same random times in both trajectories. Thus, $L$ is monotone and we may invoke Corollary \ref{co:continuous} to conclude that $\gamma_\star=-\lambda$, as desired.
\end{proof}

\paragraph{Acknowledgment.} The author warmly thanks Pietro Caputo for suggesting several interesting possible applications, which will be investigated in the near future. This work was partially supported by Institut Universitaire de France. 

\bibliographystyle{plain}
\bibliography{draft}

\end{document}